\documentclass[12pt]{amsart}
\usepackage{amssymb}
\usepackage{amsmath}
\newcommand\AND{\quad\text{and}\quad}
\newcommand\Av{\mathsf{A}}

\newcommand\C{\mathbb C}
\newcommand\DD{\mathbb D}

\newcommand\diam{\mathsf{diam}}
\newcommand\dist{\mathsf{dist}}
\newcommand\dsf{\mathsf{d}}

\newcommand\ep{\varepsilon}
\newcommand\HH{{\mathbb H}}
\newcommand\hor{\mathfrak{h}}
\newcommand\im{\mathfrak{i}\,}
\newcommand\Lap{\Delta}

\newcommand\msf{\mathsf m}
\newcommand\N{\mathbb N}

\newcommand\Prob{\mathsf{Pr}}
\newcommand\qq{\mathsf{q}}

\newcommand\R{\mathbb R}

\newcommand\Ss{\mathbb S}

\newcommand\T{\mathbb T}
\newcommand\wh{\widehat}

\newcommand\Xx{\mathbb{X}}

\numberwithin{equation}{section}

\newtheoremstyle{mythm}
  {9pt}
  {9pt}
  {\itshape}
  {0pt}
  {\bfseries}
  {}
  { }
  {\thmnumber{(#2)}\thmname{ #1}\thmnote{ #3}}

\newtheoremstyle{mydef}
  {9pt}
  {9pt}
  {\normalfont}
  {0pt}
  {\bfseries}
  {}
  { }
  {\thmnumber{(#2)}\thmname{ #1}\thmnote{ #3}}

\theoremstyle{mythm}
\newtheorem{thm}[equation]{Theorem.}
\newtheorem{pro}[equation]{Proposition.}
\newtheorem{lem}[equation]{Lemma.}
\newtheorem{cor}[equation]{Corollary.}

\theoremstyle{mydef}

\newtheorem{imp}[equation]{}

\newtheorem{rmks}[equation]{Remarks.}

\begin{document}$\,$ \vspace{-1truecm}
\title{Moments of Riesz measures on Poincar\'e disk and homogeneous tree -- a
comparative study}

\author{\bf Tetiana BOIKO and Wolfgang WOESS}
\address{\parbox{.8\linewidth}{Institut f\"ur Mathematische Strukturtheorie 
(Math C),\\ 
Technische Universit\"at Graz,\\
Steyrergasse 30, A-8010 Graz, Austria\\}}
\email{boiko@math.tugraz.at, woess@TUGraz.at}
\date{August 20, 2014, revised version + corrrection in \eqref{eq:hypLap} from June 10, 2022} 
\begin{abstract}
One of the purposes of this paper is to clarify the
strong analogy between potential theory on the open unit disk and the homogeneous
tree, to which we dedicate an introductory section. We then exemplify this analogy
by a study of Riesz measures.
Starting from interesting work by Favorov and Golinskii \cite{FaGo1}, we consider 
subharmonic functions on the open unit disk, resp. on the homogenous tree. 
Supposing that we can control the way how those functions may tend to infinity
at the boundary, we derive moment type conditions for the Riesz measures.
One one hand, we generalise the previous results of \cite{FaGo1} for the disk,
and on the other hand, we show how to obtain analogous results in the discrete
setting of the tree.  
\end{abstract}

\thanks{Supported by Austrian Science Fund projects  FWF W1230 and FWF P24028 and by NAWI Graz}


\maketitle

\markboth{{\sf T. Boiko and W. Woess}}
{{\sf Moments of Riesz measures}}
\baselineskip 15pt

\section{Introduction}\label{sec:intro}

The homogeneous tree $\T = \T_q$ with degree $q+1$ is in many respects
a discrete analogue of the hyperbolic plane. These are the two basic
examples of Gromov-hyperbolic metric spaces. In the Poincar\'e metric,
the hyperbolic plane is the open unit disk $\DD$ as a topological space. 
Its natural  geometric compactification is obtained by passing from the 
hyperbolic to the Euclidean metric and taking the closure, i.e., the 
closed unit disk. Analogously, the end compactification of $\T$ is obtained 
by passing from the original graph metric to a new (bounded) metric and 
taking the completion.

Various objects, formulas, properties, theorems, etc., of geometric, 
algebraic, analytic, potential theoretic, or stochastic nature on $\DD$ 
have counterparts on $\T$ and vice versa. It is not always immediately 
apparent that looking at $\DD$ both with Euclidean and with hyperbolic 
eyeglasses may provide additional insight. But this \emph{is} true
when one wants to understand the analogies between $\T$ and $\DD$. 
The purpose of this note is to exhibit some potential theoretic aspects 
of that correspondence. The starting point is a classical theorem of
{\sc Blaschke}~\cite{Bla}: 

\smallskip

\emph{A set $\{ z_k: k \in \N\} \subset \DD$ is the set of zeroes of a 
bounded analytic function $f$ on $\DD$ if and only if 
$\quad\displaystyle \sum_k (1-|z_k|) < \infty\,$.}

\smallskip

This also allows for the case where each $z_k$ is counted according to 
its multiplicity $\mathsf{mult}(z_k)$ as a zero of $f$. We interpet this 
theorem in terms of the subharmonic function 
$u:\DD \to [-\infty\,,\,\infty)$ given by $u(z) =\log |f(z)|$. 
We let $\mu^u$ be the Riesz measure of $u$ in its Riesz decomposition 
(see below for details). Being bounded above, $u$ has a harmonic majorant, 
which leads to finiteness of the ``moment'' 
\begin{equation}\label{eq:moment}
\int_{\DD} (1-|z|)\, d\mu^u(z)< \infty\,.
\end{equation}
Since $\mu^u = \sum_k \mathsf{mult}(z_k) \cdot \delta_{z_k}\,$, this 
just means finiteness of $\sum_k (1-|z_k|) \,\mathsf{mult}(z_k)$, 
so that the Blaschke condition takes the form \eqref{eq:moment}.  

A change of the viewpoint is now suggestive.
We start directly with a subharmonic function $u$, and instead of 
assuming that it is bounded above, we admit that it tends to $\infty$ 
in some controlled way when approaching a subset 
$E \subset \partial \DD = \Ss$, the boundary of the disk 
(the unit circle).  From properties of $E$ and the way how $u$ tends 
to infinity at $E$, we then want to deduce properties of the Riesz 
measure $\mu^u\,$. This approach was undertaken in two 
substantial papers by {\sc Favorov and Golinskii}~\cite{FaGo1}, 
\cite{FaGo2}, which were  the main inspiration for the present note.
We shall provide more general versions of some of their results on
the Riesz measure of subharmonic functions on the disk.

\smallskip

On the homogeneous tree $\T$,  the geometrical habit is converse as compared to
the disk: on one hand,
one is used to look at the \emph{Euclidean} unit disk $\DD$ and its closure, which 
in the spirit of the present note arises by a change from the ``original''
hyperbolic metric to the ``new'' Euclidean metric  which is the one of 
its compactification.
On the other hand, one is used to look at the tree with its habitual integer-valued 
graph metric -- this is our \emph{hyperbolic} object, and when we introduce the 
end compactification, we pass to a suitable new, maybe less habitual metric
which is the one that corresponds to the Euclidean metric of $\DD$.

(Sub)harmonic functions on $\T$ are defined via the discrete Laplacian $P-I$ 
(or $I-P$, if one desires a positive semidefinite operator), where $P$ is the 
transition matrix of the simple random walk. Of course, here we also have 
the Riesz decomposition theorem. We shall see that once we understand
the correspondence between tree and disk completely, we can obtain the same 
type of moment condition for the Riesz measure of a subharmonic function 
as in \eqref{eq:moment}: we need to realise that the term $1-|z|$ in 
\eqref{eq:moment} is the distance from $z$ to the boundary in the metric of 
the respective compactification.

\smallskip

In the next Section \ref{sec:disk-tree}, we provide an expository description 
of the basic potential theoretic features of $\DD$ and $\T$.
On purpose slightly beyond the scope of the subsequent section, it aims at 
providing a good  understanding of part of the many common features of those 
two structures.  
Subsequently, in Section \ref{sec:harmonic}, we present and investigate our basic
moment conditions for subharmonic functions on those two spaces.

\section{Basic potential theory on disk and tree}\label{sec:disk-tree}

 

\textbf{A. Euclidean and hyperbolic disk}

\smallskip

The Euclidean unit disk
$$
\DD = \{ z= x + \im y \in \C : |z| = \sqrt{x^2 + y^2} < 1.
$$
carries the  Euclidean metric $\dsf_{\DD}\,$, induced by the absolute value, resp. length element
\begin{equation}\label{Euclidean-metric}
\dsf_{\DD}(z,w) = |z-w| \AND d_{\DD}s =  \sqrt{dx^2 + dy^2}\,.
\end{equation}
The standard measure is Lebesgue measure -- for which here we sometimes write 
$\msf_{\DD}$ -- with area element $d_{\DD}z = dz = dx\, dy$.
The Euclidean Laplace operator is 
$$
\Lap_{\DD} = \partial_x^2 + \partial_y^2\,.
$$
A \emph{harmonic function} is a real-valued function $h \in C^2(\DD)$ such that
$\Lap_{\DD} = 0$.  For the definition of a \emph{subharmonic function}, see e.g. 
{\sc Helms}~\cite[p.58]{He}, who rather considers superharmonic functions:
the correspondence is just by a change of the sign.
A function $u : \DD \to [-\infty, +\infty)$ is
subharmonic on $\DD$ if it is upper semicontinuous, and for every
$z \in \DD$ and $r < 1-|z|$, one has $\Av_r u(z) \ge u(z)$, where
$$
\Av_r u(z) = \Av_r^{\DD}u(z) = 
\frac{1}{2\pi}\int_0^{2\pi} u(z + r e^{\im t})\,dt
$$
is the (Euclidean) average of $u$ over the circle with radius $r$
centred at $z$. In addition, we require that the set $\{z : u(z) = -\infty\}$
has Lebesgue measure $0$. 
It is well known that  $u \in C^2(\DD)$ is subharmonic
if and only if $\Lap_{\DD} u \ge 0$, see \cite[Thm. 4.8]{He}. 
If $u$ is not smooth, then $\Lap_{\DD}u$ is defined in the sense
of distributions. Subharmonicity means that this is a non-negative
Radon measure. The \emph{Riesz measure} associated with $u$ is then
\begin{equation}\label{eq:Rieszm}
\mu^u = \frac{1}{2\pi} \Lap_{\DD}u\,,\quad \text{that is,}\quad
\int_{\DD} f\, d\mu^u  = \frac{1}{2\pi} \int_{\DD} u(z)\, \Lap_{\DD}f(z)\, d\msf_{\DD}z
\end{equation}
for every $C^{\infty}$-function $f$ on $\DD$ with compact support in $\DD$. 
If $u \in C^2(\DD)$ then the ordinary function $\frac{1}{2\pi} \Lap_{\DD}u$ is
the density of $\mu^u$ with respect to Lebesgue measure $\msf_{\DD}\,$.
Furthermore, $h \in C^2(\DD)$ is harmonic if
and only if $\Av_r h(z) = h(z)$ for every $z \in \DD$ and $r < 1-|z|$.

The \emph{Green function} of $\Lap_{\DD}$ is 
\begin{equation}\label{eq:green}
G_{\DD}(z,w) = \log \frac{|1-z \bar w|}{|z-w|}\,,\quad z, w \in \DD\,.
\end{equation}
For any non-negative measure $\mu$ on $\DD$, the function $G_{\DD}\mu$ 
on $\DD$ defined by
$$
G_{\DD}\mu(z) = \int_{\DD} G_{\DD}(z,w)\,d\mu(w)\,,
$$
is called the \emph{potential} of $\mu$ if the integral is finite at some 
($\!\!\iff\!$ almost every) $z \in \DD$. Then $-G_{\DD}\mu$ is is a subharmonic 
function.
If $u$ is subharmonic and, in addition, posseses some  harmonic majorant 
on $\DD$, then
it possesses its smallest harmonic majorant $h$. In this case,
the \emph{Riesz decomposition} of $u$ has the form
\begin{equation}\label{eq:RieszD}
u = h - G_{\DD}\mu^u\,.
\end{equation}
See e.g. {\sc Ransford}~\cite[Thm. 4.5.4]{Ran}. In absence of a harmonic
majorant, for the general Riesz decomposition theorem see \cite[Thm. 3.7.9]{Ran} 
or \cite[Thm. 6.18]{He}.

\smallskip

We now consider the hyperbolic plane $\HH$. Basic hyperbolic potential theory 
appears rather to be ``common knowledge'' than being accessible in a comprehensive
treatise, with the exception of {\sc Stoll}~\cite{Sto}. See also the introductory
chapter of {\sc Helgason}~\cite{Hel}.
We use the Poincar\'e disk model; see e.g. {\sc Beardon}~\cite[Chapter 7]{Be}.
$\HH$ 
coincides with $\DD$ as a set and topologically, but the hyperbolic length 
element and metric are
\begin{equation}\label{eq:hypmetric}
d_{\HH}s =  \frac{2\sqrt{dx^2 + dy^2}}{1-|z|^2} \AND
\rho_{\HH}(z,w) = \log\frac{|1-z\bar w| + |z-w|}{|1-z\bar w| - |z-w|}.
\end{equation}
The hyperbolic measure $\msf_{\HH}$ has area element
\begin{equation}\label{eq:hypmeas}
d\msf_{\HH}(z) = d_{\HH}z =  \frac{4dz}{(1-|z|^2)^2} 
= 4 \cosh^4 \frac{\rho_{\HH}(z,0)}{2}\,dz\,.
\end{equation}
This means conversely that we can express
Lebesgue measure $\msf_{\DD}$ on $\HH$ as
\begin{equation}\label{eq:Lebmeas}
d\msf_{\DD}(z) 
= \frac{1}{4\cosh^4\bigl(\rho_{\HH}(z,0)/2\bigr)} \, d\msf_{\HH}(z)
\approx e^{-2 \rho_{\HH}(z,0)}\, d\msf_{\HH}(z)\,,\quad\text{as}\;
\rho_{\HH}(z,0) \to \infty\,.
\end{equation}
The hyperbolic Laplace operator in the variable $z = x + \im y$ is 
\begin{equation}\label{eq:hypLap}
\Lap_{\HH} = \frac{(1-|z|^2)^2}{4} \Lap_{\DD}\,.
\end{equation}
In particular, its harmonic functions are the same as the 
$\Lap_{\DD}$-harmonic functions. Above, we defined the Euclidean average
over a circle in $\DD$. Now, we let $r > 0$ and $z \in \HH$ and consider
the hyperbolic circle
$C^{\HH}(z,r) = \{w \in \HH : \rho_{\HH}(z,w) = r\}$.
This is also a Euclidean circle:
$\;
C^{\HH}(z,r) = C^{\DD}(z',r')\,,$ where
$$
z' = \frac{1-\tanh^2 (r/2)}{1- |z|^2 \tanh^2 (r/2)}\, z \AND
r'= \frac{1-|z|^2}{1- |z|^2 \tanh^2 (r/2)}\,\tanh(r/2)\,.
$$
Its 
hyperbolic length is $2\pi \sinh r$, see  \cite[page 132]{Be}.

Now, a function $u : \HH \to [-\infty, +\infty)$ is
\emph{subharmonic on $\HH$} if it is lower semicontinuous, 
$\msf_{\HH}(\{z : u(z) = -\infty\})=0$, and for every
$z \in \HH$ and $r > 0$, one has $\Av^{\HH}_r u(z) \ge u(z)$, where
$$
\Av_r^{\HH} u(z) = 
\frac{1}{2\pi \sinh r}\int_{C^{\HH}(z,r)} u\,\,d_{\HH}s\,.
$$
\begin{lem}\label{lem:euc-hyp} A function $u$ is hyperbolically superharmonic
if and only if it is superharmonic on $\DD$ in the Euclidean sense.
\end{lem}
The Green function of $\Lap_{\HH}$ is the same as the one for $\Lap_{\DD}$
given in \eqref{eq:green}, and will henceforth also be denoted by
$G_{\HH}(\cdot,\cdot)$. Using the hyperbolic metric,
\begin{equation}\label{greenhyp}
G_{\HH}(z,w) = - \log \,\tanh\bigl( \rho_{\HH}(z,w)/2\bigr).
\end{equation} 
Consequently, the hyberbolic Riesz decomposition and the Riesz measure 
of a superharmonic function $u$ are the same as the Euclidean one.

\smallskip

The natural hyperbolic \emph{compactification} $\wh \HH$ of $\HH$ arises 
from the
identification of $\HH$ with $\DD$ and taking the Euclidean closure. The 
\emph{boundary at infinity} $\partial \HH$ of $\HH$ is then the unit circle 
$\Ss$.
It is instructive to interpret this as follows: we first transform the
metric $\rho_{\HH}$ of the hyperbolic plane into a new metric, namely
the Euclidean metric. For use in the subsection on trees, note that on the
large scale, the  change of the metric is quantified by
\begin{equation}\label{eq:metrel}
\begin{aligned}
\dsf_{\DD}(z,\Ss) &= 1 - |z| = \frac{2}{1 + e^{\rho_{\HH}(z,0)}} \\ 
&\approx 2 e^{-\rho_{\HH}(z,0)} \quad \text{as} \; |z| \to 1\,,\; 
\text{or equivalently, as} \; \rho_{\HH}(z,0) \to \infty\,.
\end{aligned} 
\end{equation}
In order to get used to the two geometric views on the same object, we
can freely switch back and forth: $\DD \leftrightarrow \HH$ and
$\Ss \leftrightarrow \partial \HH$. 

\smallskip

The \emph{Poisson kernel} on $\HH \times \partial \HH = \DD \times \Ss$
is defined for $z \in \HH$, $\xi \in \Ss$ as
\begin{equation}\label{eq:poisson}
P(z,\xi) = \frac{1 - |z|^2}{|\xi - z|^2}
= \lim_{w \to \xi}\frac{G_{\HH}(z,w)}{G_{\HH}(0,w)} = e^{-\hor_{\HH}(z,\xi)}. 
\end{equation} 
with the \emph{Busemann function}
\begin{equation}\label{eq:busemann}
\hor_{\HH}(z,\xi) = 
\lim_{w \to \xi} \Bigl(\rho_{\HH}(w,z) - \rho_{\HH}(w,0)\Bigr).
\end{equation} 
It also has a probabilistic interpretation:
we start Euclidean Brownian motion (BM) at $z \in \DD$ and 
consider its hitting distribution $\nu_z$ on the boundary $\Ss$.
That is, if $B \subset \Ss$ is a Borel set, then $\nu_z(B)$ is the probability 
that the first visit of BM to $\Ss$ occurs in a point of $B$. 
Denoting by $\lambda_{\Ss}$ the normalized Lebesgue arc measure on the unit circle,
we have 
\begin{equation}\label{eq:dens-P}
\frac{d\nu_z}{d\lambda_{\Ss}} (\xi) =  P(z,\xi)\,,\quad \xi \in \Ss\,.
\end{equation}
Note that $\nu_0 = \lambda_{\Ss}\,$. 
%
%
\begin{thm}\label{thm:poissonintegral} \emph{(a)} For every $\xi \in \Ss$,
the function $z \mapsto P(z,\xi)$ is harmonic on $\DD \equiv \HH$.\\[4pt]
\emph{(b)} \emph{[Poisson representation]} For every positive harmonic 
function $h$ on $\DD \equiv \HH$,
there is a unique Borel measure $\nu^h$ on $\Ss \equiv \partial \HH$
such that 
$$
h(z) = \int_{\Ss} P(z,\cdot)\,d\nu^h\,.
$$
\emph{(c)} For every continuous function $\varphi$ on $\Ss \equiv \partial\HH$,
$$
h(z) = \int_{\Ss} P(z,\cdot)\,\varphi\, d\lambda_{\Ss}
$$
is the unique harmonic function $h$ on $\DD \equiv \HH$ such that
$$
\lim_{z \to \xi} h(z) = \varphi(\xi)\quad \text{for every}\; \xi \in \Ss\,.
$$
\end{thm}

\medskip

\textbf{B. Homogeneous tree}

\smallskip

We think of a graph as a set of vertices, equipped with a symmetric
neighbourhood relation $\sim$. An edge is a pair (usually considered 
un-oriented) $e=[x,y]$ with $x \sim y$. 
Now, we consider the homogeneous tree $\T = \T_{\qq}\,$, where every vertex
has $\qq+1 \ge 3$ neighbours. The discrete Laplacian $\Lap_{\T}$ acts on
functions $f: \T \to \R$ by
\begin{equation}\label{eq:discLap}
\Lap_{\T}f(x) = \frac{1}{\qq+1}\sum_{y \sim x} \bigl(f(y) - f(x)\bigr)\,.
\end{equation}
It is related with \emph{simple random walk} (SRW) on $\T$ in the same way as
the above Laplacians are the infinitesimal generators of Euclidean and
hyperbolic \emph{Brownian motion}. (Hyperbolic BM is Euclidean BM slowed
down close to the boundary of the hyperbolic disk.) SRW  is
the Markov chain $(Z_n)_{n \ge 0}$ on $\T$ where $Z_n$ is the random position 
at discrete time $n$ of the particle, which moves from the
current vertex $x$ to any of its neighbours $y$ with equal probability 
$p(x,y) =1/(\qq+1)$, while $p(x,y) =0$ if $x \not\sim y$. This gives rise
to the transition operator $Pf(x) = \sum_y p(x,y)f(y)$, and 
$\Lap_{\T} = P - I$, where $I$ is the identity operator.
For potential theory on trees, see e.g. {\sc Woess}~\cite{W-Markov}.

\smallskip

For any pair of vertices $x,y$, there is a \emph{geodesic path} $\pi(x,y)$
from $x$ to $y$ without repetitions. The number of edges of that path is
the graph distance
\begin{equation}\label{eq:tree-metric}
\dsf_{\T}(x,y)\,.
\end{equation}
The standard measure on (the vertex set of) $\T$ is the counting measure
$\msf_{\T}(A) = |A|$ ($A \subset \T$).
In comparing $\T$ with $\HH$, $\dsf_{\T}$ and $\msf_{\T}$ correspond to the 
hyperbolic distance and area element
$\rho_{\HH}$ and $\msf_{\HH}$ on $\HH$, respectively.

\smallskip

Functions $h, u: \T \to \R$ are \emph{harmonic}, resp. \emph{subharmonic},
if $\Lap_{\T}h = 0$, resp.  $\Lap_{\T}u \ge 0$.  
Equivalently, $Ph =h$, resp. $Pu \ge u$, a definition in terms of the arithmetic 
averages over spheres with radius $1$.
Here, it makes no sense
to allow for value $-\infty$, since sets of $\msf_{\T}$-measure $0$ are empty.

The Green function of $\Lap_{\T}$ is
\begin{equation}\label{eq:Green-T}
G_{\T}(x,y) = \frac{\qq}{\qq-1}\, \qq^{-\dsf_{\T}(x,y)}\,,\quad x,y \in \T\,.
\end{equation} 
The potential of a non-negative real function $f$ on $\T$ is
$$
G_{\T}f(x) = \sum_{y \in \T} G_{\T}(x,y) \,f(y)\,.
$$
Since $\T$ is countable, measures on $\T$ are defined by their atoms, that
is, they can be identified with non-negative functions. Thus, the Riesz measure
of a subharmonic function $u$ can be identified with the function  
\begin{equation}\label{eq:RieszT}
\mu^u = \Lap_{\T} u = Pu - u\,.
\end{equation}
More precisely, the function $Pu-u$ should be understood as the density of the 
Riesz measure $\mu^u$ with respect to the counting measure $\msf_{\T}\,$.
%
If $u$ has a harmonic majorant, then its Riesz decomposition reads
$$
u = h - G_{\T}\mu^u\,,  \quad \text{where} \quad h(x) = \lim_{n \to \infty} P^n u(x)
$$ 
is the smallest harmonic majorant of $u$. Again, if there is no harmonic majorant, then
the statement of the Riesz decomposition theorem is a little bit more involved. In  the
Markov chain (= Discrete Potential Theory) literature, the only source for the latter
seems to be \cite{W-Riesz}, but it will not be needed here.

\smallskip

We  next describe the \emph{end compactification} of $\T$. 
A \emph{ray} or \emph{geodesic ray} is a one-sided infinite sequence
$\pi= [x_0, x_1, x_2,\dots]$ of vertices such that $x_n \sim x_{n-1}$ and $x_n \ne x_m$  
for all $n, m$, $n \ne m$. Two rays are called \emph{equivalent}, if 
they differ only by finite initial pieces. An end
of $\T$ is an equivalence class of rays. The set of all ends is the boundary $\partial \T$.
We set $\wh \T = \T \cup \partial \T$. 
For every vertex $x \in \T$ and  every $\xi \in \partial \T$, there is precisely 
one geodesic ray $\pi(x,\xi)$ starting at $x$ that represents $\xi$. 
Analogously, for any two disctinct ends $\xi, \eta$, there is a unique two-sided 
\emph{geodesic} $\pi(\xi,\eta) = [\dots, -x_2, -x_1, x_0, x_1, x_2, \dots]$ such that
$[x_k, x_{k-1}, x_{k-2}, \dots]$ and $[x_k, x_{k+1}, x_{k+2}, \dots]$ are rays representing
$\xi$ and $\eta$, respectively. 

We now pursue the line followed above by an exponential change of the metric of $\HH$, see 
\eqref{eq:metrel}. A natural choice is as follows.  We fix a root vertex $o \in \T$. 
For $z \in \wh\T$, we denote $|z| = \dsf_{\T}(o,z)$, with value $\infty$ if $z \in \partial \T$. 
For $w, z \in \wh\T$ , we define their \emph{confluent}
$w \wedge z$ with respect to $o$ as the last common element on the
geodesics $\pi(o,w)$ and $\pi(o,z)$. This is a vertex, unless $z=w \in \partial \T$.
We let
\begin{equation}\label{eq:ultra}
\rho_{\T}(w, z) = \begin{cases}q^{-|w \wedge z|} \,,&\text{if}\; z\ne w\,,\\
                                              0 \,,&\text{if}\; z=w\,.
\end{cases}
\end{equation}
This is an ultra-metric. In the induced topology, $\wh \T$ is compact, and $\T$ is
discrete and dense. Convergence in this topology is as follows: if $\xi \in \partial \T$
then a sequence $(z_n)$ in $\wh \T$ converges to $\xi$ if and only if 
$|\xi \wedge z_n| \to \infty\,$. 


At this point, we underline that in the ``translation'' from disk to tree,
the graph metric $\dsf_{\T}$ corresponds to the hyperbolic metric 
$\rho_{\HH}\,$, while the metric $\rho_{\T}$ is the one that may be interpreted
to correspond to the Euclidean metric $\dsf_{\DD}\,$. 
The next identity should be compared with \eqref{eq:metrel}.
\begin{equation}\label{eq:metrel1}
\rho_{\T}(x, \partial \T) = \qq^{-|x|}
\quad\text{for}\; x \in \T\,.
\end{equation}

\smallskip

The \emph{Martin kernel} on $\T \times \partial \T$
is defined for $x \in \T$, $\xi \in \partial\T$ as
\begin{equation}\label{eq:martin}
K(x,\xi)
 = \lim_{y\to \xi}\frac{G_{\T}(x,y)}{G_{\T}(o,y)} = \qq^{-\hor_{\T}(x,\xi)}
\end{equation} 
with the \emph{Busemann function}
\begin{equation}\label{eq:busemann1}
\hor_{\T}(x ,\xi) = 
\lim_{y \to \xi} \Bigl(\dsf_{\T}(y,x) - \dsf_{\T}(y,o)\Bigr) = 
\dsf_{\T}(x \wedge \xi,x) - \dsf_{\T}(x \wedge \xi,o).
\end{equation} 
Again, we have a probabilistic interpretation. It is a well-known exercise to
show that SRW on $\T$ converges alsmost surely in the
topology of $\wh \T$ to a limit random variable  $Z_{\infty}$ that takes its values in
$\partial \T$. Let $\nu_x$ be the distribution of $Z_{\infty}\,$, when SRW starts
at vertex $x$. Then $\nu_o = \lambda_{\partial \T}$ is the tree-analogue of
the normalized Lebesgue measure $\lambda_{\Ss}$ on the unit circle:
$\lambda_{\partial \T}$ is the unique probability measure on
$\partial\T$ which is invariant under ``rotations'' of $\T$, that is,
self-isometries of the graph $\T$ which fix the root vertex $o$. 
Connectedness of $\T$ implies that
$\nu_x$ is absolutely continuous with respect to $\lambda_{\partial \T}\,$, and
the Radon-Nikodym-derivative is (realised by) the Martin kernel:
\begin{equation}\label{eq:dens-K}
\frac{d\nu_x}{d\lambda_{\partial \T}}(\xi) = K(x,\xi)\,.
\end{equation}
We have a perfect analogy with Theorem \ref{thm:poissonintegral}.

\begin{thm}\label{thm:poissonmartin} \emph{(a)} For every $\xi \in \partial\T$,
the function $x \mapsto K(x,\xi)$ is harmonic on $\T$.\\[4pt]
\emph{(b)} For every positive harmonic function $h$ on $\T$,
there is a unique Borel measure $\nu^h$ on $\partial \T$
such that 
$$
h(x) = \int_{\partial \T} K(x,\cdot)\,d\nu^h\,.
$$
\emph{(c) [Solution of the Dirichlet problem]} 
For every continuous function $\varphi$ on $\partial\T$,
$$
h(x) = \int_{\partial \T} \varphi\, d\nu_x
= \int_{\partial \T} K(x,\cdot)\,\varphi\, d\lambda_{\partial\T}\,.
$$
is the unique harmonic function $h$ on $\T$ such that
$$
\lim_{x \to \xi} h(x) = \varphi(\xi)\quad \text{for every}\; \xi \in \partial \T\,.
$$
\end{thm}

\medskip

\textbf{C. A table of correspondences}\label{sec:table}

As a general ``rule'' of translation from $\HH$ to $\T$, we note that
base $e$ (Eulerian number) has to be replaced by base $\qq$ (branching
 number of the tree).
 
\bigskip

\begin{tabular}{ l | c | l | c   }
hyperbolic plane $\HH$ & ref. number & homogeneous tree $\T$ & ref. number
\\[4pt]
\hline
&&&
\\[-16pt]
\hline
&&&
\\[-8pt]
hyperbolic metric $\rho_{\HH}$ & \eqref{eq:hypmetric} & 
graph metric $\dsf_{\T}$ & \eqref{eq:tree-metric} 
\\[4pt]
\hline
&&&
\\[-8pt]
Euclidean metric $\dsf_{\DD}$ & \eqref{Euclidean-metric} &
length metric $\rho_{\T}$ & \eqref{eq:ultra}
\\[8pt]
\hline
&&&
\\[-8pt]
boundary $\Ss$ (unit circle)& & boundary $\partial \T$ & 
\\[4pt]
\hline
&&&
\\[-8pt]
${\displaystyle \text{compactification} \atop 
\displaystyle \wh \HH = \DD^- = \DD\cup\Ss}$ & & 
${\displaystyle \text{compactification} \atop 
\displaystyle \wh \T = \T \cup \partial \T}$ & 
\\[8pt]
\hline
&&&
\\[-8pt]
hyperbolic measure $\msf_{\HH}$ &  \eqref{eq:hypmeas} &
counting measure $\msf_{\T}$ & 
\\[4pt]
\hline
&&&
\\[-8pt]
Lebesgue measure $\msf_{\DD}$ & \eqref{eq:Lebmeas} & $\qq^{-2|x|}\,d\msf_{\T}(x)$ 
& \eqref{eq:Lebmeas}
\\[4pt]
\hline
&&&
\\[-8pt]
${\displaystyle \text{normalised arc measure}} \atop 
{\displaystyle \lambda_{\Ss} \text{ on } \Ss}$& 
&
${\displaystyle \text{rotation invariant measure}} \atop 
{\displaystyle \lambda_{\partial\T} \text{ on } \partial\T}$& 
%
\\[4pt]
\hline
&&&
\\[-8pt]
hyperbolic Laplacian $\Lap_{\HH}$ & \eqref{eq:hypLap}& 
discrete Laplacian $\Lap_{\T}$ 
& \eqref{eq:discLap}
\\[4pt]
\hline
&&&
\\[-8pt]
Green function $G_{\HH} = G_{\DD}$ & \eqref{eq:green}& 
Green function $G_{\T}$ 
& \eqref{eq:Green-T}
\\[4pt]
\hline
&&&
\\[-8pt]
Poisson kernel & \eqref{eq:poisson}& 
Martin kernel 
& \eqref{eq:martin}
\\[4pt]
\hline
  \end{tabular}

\bigskip

There are many further analogies between analysis, probability, group actions, etc.
on $\DD$ and $\T$. The present introduction is not intended to cover all those aspects.
For further tips of the iceberg, see e.g. {\sc Casadio Tarabusi, Cohen, Kor\'anyi and 
Picardello}~\cite{CCKP}, {\sc Rigoli, Salvatori  and Vignati}~\cite{RSV}, 
{\sc Cohen, Colonna and Singman}~\cite{CCS1}, 
{\sc Atanasi and Picardello}~\cite{AP} or 
{\sc Casadio Tarabusi and Fig\`a-Talamanca}~\cite{CF}, and the references given there. 

\section{Moment conditions and harmonic majorants}\label{sec:harmonic}

Let $\Xx = \DD$ or $\Xx = \T$, with respective boundary $\partial \Xx$ and compactification $\wh \Xx$ 
(see the above table). The boundary carries the metric $\dist$ and measure $\lambda$, where 
$\dist = \dsf_{\DD}$ and $\lambda = \lambda_{\Ss}$ in case of the disk, 
while $\dist = \rho_{\T}$ and $\lambda = \lambda_{\T}$
in case of the tree. Given a subharmonic function $u$ on $\Xx$
and its Riesz measure $\mu^u$, we are interested in finiteness of its \emph{first (boundary) moment}
\begin{equation}\label{eq:momX}
\int_{\Xx} \dist(x,\partial \Xx)\, d\mu^u(x) 
\end{equation}
and variants thereof.
One principal tool is the following lemma.
\begin{lem}\label{lem:momX}
The subharmonic function $u$ has a harmonic majorant on $\Xx$ if and only if $\mu^u$ has finite 
first moment \eqref{eq:momX}. 
\end{lem}

\begin{proof}
Our function $u$ has a harmonic majorant if and only if $G_{\Xx}\mu^u$ is a potential,
that is, it is finte at some $x \in \Xx$. 

\smallskip

If $\Xx = \DD$ then {\sc Armitage and Gardiner}~\cite[Thms. 4.2.4 and 4.2.5]{AG}
show that $G_{\Xx}\mu^u$ is a potential if and only if \eqref{eq:momX} holds.

\smallskip
   
If $\Xx = \T$ then by \eqref{eq:Green-T}, $G_{\T}(x,y) \le q^{|x|} G_{\T}(o,y)$ for all $x,y$, so that 
$G_{\T}\mu^u$ is finite at $x \in \T$ if and only if $G_{\T}\mu^u(o) < \infty$. Now
$$
G_{\T}\mu^u(o) = \sum_{x \in \T} \frac{q}{q-1} q^{-|x|} \mu^u(x)
= \frac{q}{q-1} \int_{\T} \rho_{\T}(x, \partial \T) \, d\mu^u(x)\,. \eqno{\qedhere}
$$   
\end{proof}
So in fact what we are going to do is to exhibit a sufficient condition for
a subharmonic function on $\Xx=\DD$, resp. $\Xx = \T$, to possess a (global or restricted)
harmonic majorant, even if it is not bounded above. 

\begin{thm}\label{thm:main1} Let $u$ be a subharmonic function on $\Xx$
and consider the closed set
$$
E = \Bigl\{ \xi \in \partial \Xx : \limsup_{\Xx \ni x \to \xi} u(x) = \infty \Bigr\}.
$$
Suppose that $\Psi: [0\,,\,\diam(\Xx)] \to [0\,,\,\infty]$ is 
a continuous, decreasing function with 
$$
\Psi(t) = \infty \iff t=0 \AND \lim_{t\to 0} \Psi(t) = \infty\,,
$$
and that
$$
u(x) \le \Psi \bigl(\dist(x,E)\bigr) \quad \text{for all}\; x \in \Xx\,.
$$
If 
\begin{equation}\label{eq:fin-int}
\int_{\partial\Xx} \Psi \bigl(\dist(\xi,E)\bigr) \,d\lambda(\xi) < \infty\,,
\end{equation}
then $u$ has a finite harmonic majorant, and the Riesz measure $\mu^u$ has finite first 
boundary moment.
\end{thm}

We note that for condition \eqref{eq:fin-int} it is necessary 
that $\lambda(E)=0$. 
For the proof of the theorem, we shall work with the function
\begin{equation}\label{eq:h}
h = \int_{\partial \Xx} K_{\Xx}(\cdot,\xi)\,  \Psi \bigl(\dist(\xi,E)\bigr) \,d\lambda(\xi)\,,
\end{equation}
where $K_{\Xx}$ is the Poisson kernel \eqref{eq:poisson} when $\Xx = \DD$, and the Martin 
kernel \eqref{eq:martin} when $\Xx = \T$. Since for fixed $x \in \Xx$, the function
$\xi \mapsto K_{\Xx}(x,\xi)$ is continuous on $\partial \Xx$ (whence bounded), 
the function $h$ is finite and harmonic on $\Xx$ under condition \eqref{eq:fin-int}. 

We need some preparations. 
We let $0 < t \le \max\{\dist(x,E) : x \in \Xx\}$ and consider the sets 
$$
E^{(t)} = \{ \xi \in \partial\Xx: \dist(\xi, E) \le t\} \AND
E^{(t)}_* = \{ \xi \in \partial\Xx: \dist(\xi, E) > t\}\,,
$$
and, for $0 < t < 1$, the set  $\Xx^{(t)}$ which is the component of the origin of the set
$\{ x \in \Xx : \dist(x,E) > t \}\,.$
\\[5pt]
\textbf{Disk case:} $\DD^{(t)}$ (denoted $\Omega_t$ in \cite{FaGo1}) is an open domain, 
and its boundary is
$$
\partial \DD^{(t)} = \partial_{\infty}\DD^{(t)} \,\cup\, \Gamma^{(t)}, 
\quad\text{where }\; \partial_{\infty}\DD^{(t)} \subset \overline{E^{(t)}_*} \;\text{ and }\; 
\Gamma^{(t)} \!= \Gamma^{(t)}_{\DD} = \{ z \in \DD: \dsf_{\DD}(z, E) = t\}.
$$
The sets $E^{(t)}$ and $\partial_{\infty}\DD^{(t)}$ are both unions of finitely many closed arcs 
on $\Ss$ and meet at finitely many endpoints of those arcs. $\partial_{\infty}\DD^{(t)}$ 
may be a strict subset of the closure of $E^{(t)}_*\,$, because some arcs of the latter
set can be the boundary of a different component of 
$\{ z \in \DD : \dsf_{\DD}(z,E) > t \}$. (The latter can arise as ``triangular'' regions 
bounded by an arc of $\Ss$ and of arcs of two intersecting circles $\{ z : |z-\zeta_j| = t\}$,
where $\zeta_j \in E$, $j=1,2$.) 
\\[5pt]
\textbf{Tree case:} The origin is of course the root vertex of $\T$. The metric 
$\dist = \rho_{\T}$
takes only the countably many values $\qq^{-k}$, $k \ge 0$ (integer). For $0 < t < 1$
let $k \ge 1$ be the integer such that 
\begin{equation}\label{eq:kt}
\qq^{-k} \le t < \qq^{-(k-1)}\,,\quad k=k(t).
\end{equation}
For any vertex $y \in \T$, we consider the \emph{branch} of $\T$
at~$y$. This is the subtree (induced by)
$$
\mathrm{T}_y = \{ u \in \T : y \in \pi(o,u) \}.
$$
Its boundary $\partial \mathrm{T}_y \subset \partial \T$ consists of those ends which are
represented by geodesics that lie entirely within $\mathrm{T}_y\,$. Note that the open-compact 
sets $\partial \mathrm{T}_y\,$, $y \in \T$, are a basis of the topology of $\partial \T$.  
Given $t$, let $k = k(t)$ and consider the set
$$
\Gamma^{(t)} = \Gamma^{(t)}_{T} =
\{ y \in \T : |y| = k \,,\; \partial \mathrm{T}_y \cap E \ne \emptyset \}\,.
$$
We have 
$$
E^{(t)} = E^{(t)}_{\T} = \bigcup _{y \in \Gamma^{(t)}} \partial \mathrm{T}_y\,.
$$%
For small $t$ $\equiv$ large $k=k(t)\,$, only few vertices $y$ with $|y|=k$ belong to $\Gamma^{(t)}$: 
as $t \to 0$ $\equiv$ $k \to \infty$, 
we have
$$
\frac{|\Gamma^{(t)}|}{|\{ y \in \T : |y| = k(t)\}|} = 
\lambda_{\partial \T}\bigl(E^{(t)}\bigr) 
\to \lambda_{\partial \T}(E)=0.
$$
When $\Xx=\T$, the set $\T^{(t)}$ is the subtree of $\T$ obtained 
by chopping off each branch $\mathrm{T}_y\,$, $y \in \Gamma^{(t)}$, that is,
$$
\T^{(t)} = \T \setminus \bigcup\nolimits_{y \in \Gamma^{(t)}} \mathrm{T}_y\,.
$$
The \emph{boundary} of this truncated tree is
$$
\partial\T^{(t)} =  \partial_{\infty}\T^{(t)} \cup \Gamma^{(t)}_{\T}\,,\quad \text{where}\quad
\partial_{\infty}\T^{(t)} = E^{(t)}_*\,,
$$ 
while  $\Gamma^{(t)}$ is the outer vertex boundary of $\T^{(t)}$: it consists of 
those vertices in the complement that have a neighbour (here: precisely one 
neighbour) in $\T^{(t)}$. In the topology of $\wh \T$, we have the compact subspaces 
$\wh \T^{(t)} =  \T^{(t)} \cup \partial \T^{(t)}$ and the boundary  
$\partial \T^{(t)}$. 

\smallskip

We shall need the following simple estimate.

\begin{lem}\label{lem:nuw}
For $x \in \Xx$, consider the harmonic measure $\nu_x$ on $\partial \Xx\,$; see
\eqref{eq:dens-P}, resp. \eqref{eq:dens-K}. Then
$$
\text{for }\; y \in \Gamma^{(t)}\,,\quad  \nu_y(E^{(t)}) \ge 1/c_{\Xx} \;=
\begin{cases}
1/3\,, &\text{if} \;\ \Xx=\DD\,,\\ 
\qq/(\qq+1)\,, &\text{if} \;\ \Xx=\T\,.
\end{cases}
$$
\end{lem}

\begin{proof}
$\quad$
\textbf{A. Disk case.} 
For $y \in \Gamma^{(t)}_{\DD}$ there is $\zeta = \zeta_y \in E$ such that 
$|y - \zeta_y| = \dsf(y,E) = t$. Consider the arc 
$\gamma_{\zeta} = \{\xi \in \Ss: |\xi - \zeta| \le t \} \subset E^{(t)}$,
as well as the circle $\{ z \in \C : |z - \zeta| = t \}$. At any of the two 
intersection points  of that circle with $\Ss$, the angle $\alpha$ between the tangents 
to the two circles is such that $\pi/2 > \alpha > \pi/3$, as $0 < t < 1$. 
By \cite[p. 13, Fig.1.1]{Gar},  $\nu_y(\gamma_{\zeta})= \alpha/\pi > 1/3$. 
(In \cite{FaGo1}, the lower estimate $1/6$ is used, but apparently also $1/3$ works.)
\\[5pt]
\textbf{B. Tree case.} 
For $y \in \Gamma^{(t)}_{\T}$, we have that $\partial \mathrm{T}_y \subset E^{(t)}$.
We note that $\nu_y$ gives equal mass to the boundaries of each of the $\qq+1$ branches of $\T$
that are emanating from $y$. Among those, $\qq$ branches are part of $\mathrm{T}_y\,$, that
is, $\nu_y(\partial \mathrm{T}_y) = \qq/(\qq+1)$, providing the lower bound. 
\end{proof}

\begin{proof}[\bf Proof of Theorem \ref{thm:main1}]
Consider the function 
$\psi^{(t)}(\xi) = \min \bigl\{ \Psi(t), \Psi\bigl( \dist(\xi,E)\bigr)\bigr\}$
on $\partial \Xx$ and the harmonic function 
$$
h^{(t)}(x) = \int_{\partial \Xx} K(x,\cdot)\, \psi^{(t)}\, d\lambda
= \int_{\partial \Xx}\psi^{(t)}\, d\nu_x\,.
$$
We know from theorems \ref{thm:poissonintegral}, resp. \ref{thm:poissonmartin} that 
it is the solution of the Dirichlet problem on $\Xx$ with continuous boundary function 
$\psi^{(t)}$. We have 
$\psi^{(t)}(\xi) = \Psi(t)$ on $E^{(t)}$, while $\psi^{(t)}(\xi) \le \Psi(t)$
on $E^{(t)}_* \supset \partial_{\infty}\Xx^{(t)}$. Thus,
\begin{equation}\label{eq:ht}
h^{(t)}(x) = \int_{E^{(t)}_*}\Psi\bigl( \dist(\cdot,E)\bigr)\, d\nu_x
+ \Psi(t) \, \nu_x(E^{(t)})\,.
 \end{equation}
Taking boundary limits for points $x$ within $\wh\Xx^{(t)}$, and using Lemma \ref{lem:nuw},
\begin{equation}\label{eq:lim-h} 
\begin{aligned}
\lim_{x \to \xi} h^{(t)}(x) &= \Psi\bigl( \dist(\xi,E)\bigr)\,,
\quad\text{for }\;\xi \in \partial_{\infty}\Xx^{(t)}\,,\AND\\
\lim_{x \to y} h^{(t)}(x) &= h^{(t)}(y) 
\ge \Psi(t)\, \nu_y(E^{(t)}) \ge \Psi(t)/c_{\Xx} 
\quad\text{for }\;y \in \Gamma^{(t)}.
\end{aligned}
\end{equation} 
(In the tree case, since $y$ is an isolated point, the last limit just means stabilisation at $y$.) 
On the other hand, by assumption our subharmonic function $u$ satisfies 
\begin{equation}\label{eq:lim-u} 
\begin{aligned}
\limsup_{x \to \xi}u(x) &\le \Psi\bigl( \dist(\xi,E) \bigr) 
\quad\text{for}\; \xi \in \partial_{\infty}\Xx^{(t)}\,,\AND \\
\limsup_{x \to y} u(y) &\le \Psi(t) \quad\text{for}\; y \in \Gamma^{(t)}\,.
\end{aligned}
\end{equation}
Therefore, again taking boundary limits within $\wh\Xx^{(t)}$, 
$$
\limsup_{x \to \eta} \Bigl(u(x) - c_{\Xx}\, h^{(t)}(x)\Bigr) \le 0 \quad \text{for every}\;
\eta \in \partial \Xx^{(t)}.
$$
Thus, by the maximum principle (which also holds on the tree because $\T^{(t)}$ is a connected graph,
a simple excercise),
\begin{equation}\label{eq:u-le-ht}
u(x) \le c_{\Xx}\, h^{(t)}(x) 
\quad  \text{for every}\; x \in \Xx^{(t)}.
\end{equation}
Having this, we obtain the proposed first moment: 
let  $h(x) = \int_{\partial \Xx} K(x,\cdot) \,\Psi\bigl( \dist(\cdot,E) \bigr) \,d\lambda$
be the harmonic function proposed in \eqref{eq:h}. Then $h^{(t)} \le h$ on $\Xx^{(t)}$ for any
$t$. Given any $x \in \Xx$, we can choose $t < \dist(x,E)$ to see that 
$c_{\Xx}\cdot  h$ is a (finite) harmonic majorant for our subharmonic function~$u$. 
\end{proof}

There is a simple converse to Theorem \ref{thm:main1}.

\begin{pro}\label{pro:converse} Let $u$ be a subharmonic function on $\Xx$, and let
$E$ and $\Psi$ be as in Theorem \ref{thm:main1}. If
$$
u(x) \ge \Psi \bigl(\dist(x,E)\bigr) \quad \text{for all}\; x \in \Xx
$$
and
\begin{equation}\label{eq:infin-int}
\int_{\partial\Xx} \Psi \bigl(\dist(\xi,E)\bigr) \,d\lambda(\xi) = \infty
\end{equation}
then $u$ has no harmonic majorant on $\Xx$, and the first moment
of $\mu^u$ is infinite.
\end{pro}

\begin{proof}
We give a combined proof for $\Xx=\DD$ and $\Xx=\T$.
Suppose that the first moment of $\mu^u$ is finite. Then by Lemma \ref{lem:momX},
$u$ has a (finite) harmonic majorant $h$. Consider the continuous function
$\Psi_M = \min\{\Psi, M\}$. Then for all $x \in \Xx$,
$$
h(x) \ge u(x) \ge \Psi_M \bigl(\dist(x,E)\bigr)
$$
The function 
$$
g_M(x) = \int_{\partial \Xx} K_{\Xx}(\cdot,\xi)\,  \psi_M \bigl(\dist(\xi,E)\bigr) \,d\lambda(\xi)\,,
$$
defined analogously to  \eqref{eq:h}, provides the solution of the Dirichlet problem
on $\Xx$ with boundary data $\psi_M \bigl(\dist(\xi,E)\bigr)$. 
We have 
$$
\liminf_{x \to \xi} \bigl(  h(x) - g_M(x) \bigr) \ge 0 \quad \text{for every }\; 
\xi \in \partial \Xx\,.
$$
By the minimum principle, $h \ge g_M$ on $\Xx$, and in particular, $h(o) \ge g_M(o)$.
Letting $M \to \infty$, monotone convergence yields $h(o)=\infty$,
contradicting finiteness of $h$.
\end{proof}

Next, in a similar spirit to \cite{FaGo1}, we want to extend Theorem \ref{thm:main1} to a 
situation where the integral in \eqref{eq:fin-int} is infinite.
For that purpose,  we shall need an estimate of the Green function 
$G_{\Xx^{(t)}}(x,y)=G_{\Xx^{(t)}}(y,x)$ of $\Xx^{(t)}$. On the disk, this function
is of course well described in the classical potential theory literature.

On the tree, for $x,y \in \T^{(t)}$, it is the expected number of
visits to $y$ of the random walk starting at $x$ before it hits $\Gamma^{(t)}$.
It is natural to define $G_{\T^{(t)}}(x,y)= 0$ when one of $x, y$ lies in $\Gamma^{(t)}$ 
and the other in $\T^{(t)}$. In potential theoretic terms,
$f= G_{\T^{(t)}}(\cdot,y)$ is the smallest non-negative function on $\T^{(t)} \cup \Gamma^{(t)}$ 
satisfying $\Delta_{\T}f(x) = -\delta_y(x)$ for $x \in \T^{(k)}$. This corresponds directly
to the disk situation.


\begin{thm}\label{thm:green}
Define $r=r_{\Xx}\,$, $a = a_{\Xx}$ and $b=b_{\Xx}$ for $\Xx=\DD$ or $=\T$ by
$$
r_{\DD} = 7\AND a_{\DD}=b_{\DD}= 18\,,\quad\text{resp.}\quad r_{\T} = 1\,,\; a_{\T} = \qq/(\qq-1) \AND
b_{\T} =1\,.
$$
Let $0 < t < 1/r$. Then for any $x \in \Xx^{(rt)}$, we have
$$
G_{\Xx}(x,o) \ge G_{\Xx^{(t)}}(x,o) \ge \frac{1}{a}\, G_{\Xx}(x,o) \ge \frac{1}{b}\,\dist(x,\partial X)\,,
$$
where $o$ is the origin (root) of $\Xx$.
\end{thm}
\begin{proof}
The first inequality is clear in both cases. The third inequality is also clear, and
it is an equality in the tree case. We need to prove the second 
inequality separately for tree and disk, and begin this time with the tree. 
\\[5pt]
\textbf{A. Tree case.} 
Let $\nu_x^{(t)}$ be the harmonic measure of $\T^{(t)}$ on its boundary. In particular, for 
$y \in \Gamma^{(t)}$, the probability that the random walk starting at $x$ first hits 
$\Gamma^{(t)}$ in $y$ is $\nu_x^{(t)}(y)$. 
The function $g^{(t)}(x) = G_{\T}(x,o) - G_{\T^{(k)}}(x,o)$ is positive
harmonic on $\T^{(t)}$. We have $\lim_{x \to \xi} g^{(t)}(x) = 0$ for $\xi \in \partial_{\infty}\T^{(t)}$
(because this holds for $G_{\T}(x,o)$), while $g^{(t)}(y) = G_{\T}(y,o)$ for $y \in \Gamma^{(t)}$. 
Since the Dirichlet problem on $\widehat \T^{(t)}$
admits solution (a straightforward adaptation of \cite[Thm.4]{CSW}, including in that argument 
vertices which are boundary points), we get that
$$
g^{(t)}(x) = \sum_{y \in \Gamma^{(t)}} G_{\T}(y,o) \nu_x^{(t)}(y) 
= \frac{\qq}{\qq-1}\qq^{-k} \,\nu_x^{(t)}(\Gamma^{(t)})\,, 
$$ 
where $k=k(t)$, as defined in \eqref{eq:kt}. In the last identity (which can of course also
be derived probabilistically), \eqref{eq:Green-T} was used. Now let $x \in \T^{(t)}$ and
let $x_0$ be the last point on the geodesic $\pi(o,x)$ that lies on some $\pi(o,y)$ with
$y \in \Gamma^{(t)}$. Note that $|x_0| \le k-1$. In order to reach $\Gamma^{(t)}$, the random 
walk starting at $x$ needs to pass through $x_0$. Unless $x = x_0$, this is unrestricted random 
walk on $\T$ before the first visit in $x_0$, because up to that time it evolves on a 
branch of $\T$ that contains no element of $\Gamma^{(t)}$. It is well known and easy to see that
$$
\Prob[ \exists n : Z_n = x_0 \mid Z_0 = x] = G(x,x_0)/G(x_0,x_0),
$$ 
see e.g. \cite[Thm.1.38]{W-Markov}. Thus (compare with \cite[Prop.9.23]{W-Markov}), 
$$
\nu_x^{(t)}(\Gamma^{(t)}) = \Prob[ \exists n : Z_n = x_0 \mid Z_0 = x] 
\underbrace{\nu_{x_0}^{(t)}(\Gamma^{(t)})}_{\displaystyle \le 1} 
\le \qq^{-\dsf_{\T}(x,x_0)} = \qq^{|x_0|-|x|} \le \qq^{k-1-|x|}\,.
$$ 
We infer that
$$
g_k(x) \le \frac{\qq}{\qq-1}\qq^{-k}\qq^{k-1-|x|} = \frac{1}{\qq-1}\qq^{-|x|}
$$
Consequently, 
$$
G_{\T^{(k)}}(x,o) = G_{\T}(x,o)-g_k(x) = \frac{\qq}{\qq-1}\qq^{-|x|} -g_k(x) \ge \qq^{-|x|}\,, 
$$
and in view of \eqref{eq:Green-T}, the proposed estimate is proved for the tree.
\\[5pt]
\textbf{B. Disk case.} 
The proof follows \cite{FaGo1}, but we re-elaborate
it to get the constant $a_{\DD} = 7$ and to have $G_{\DD}(z,0)$ in the lower bound.
As before, we prefer to write $z$ instead of $x$ for the
elements of $\DD$. We start in the same way as for the tree.
We know that $G_{\DD}(z,0)= \log \frac{1}{|z|}\,$, and we can decompose
$$
G_{\DD^{(t)}}(z,0) = G_{\DD}(z,0) - g^{(t)}(z)\,, \quad z \in \DD^{(t)}\,,
$$
where $g^{(t)}$ is harmonic on $\DD^{(t)}$ with boundary values $0$ at
$\partial_{\infty}\DD^{(t)}$. For $z \in \Gamma^{(t)}$, there is 
$\zeta \in E$ with $|z-\zeta| = t$, whence $|z| \ge 1-t$. Thus, using
\eqref{lem:nuw}, 
\begin{equation}\label{eq:gtbound}
g^{(t)}(z) = G_{\DD^{(t)}}(z,0) \le \log \frac{1}{1-t} \le 3\,\log\frac{1}{1-t}\, \nu_z(E^{(t)}).
\end{equation}
The right hand side is a harmonic function of $z$ on the whole of $\DD$.
By the maximum principle, \eqref{eq:gtbound} holds on all of $\DD^{(t)}$.

We now choose real parameters $r > s > 1$ with $r-s > 1$. We assume that $t < 1/r$.
Let $z \in \DD$.

\noindent\emph{Case 1.} Let $|z| < (1-t)^s$. Then 
$g^{(t)}(z) \le \log \frac{1}{1-t} \le \frac{1}{s}\log \frac{1}{|z|}$, and
$$
G_{\DD^{(t)}}(z,0) \ge \frac{s-1}{s} G_{\DD}(z,0)\,.
$$
\emph{Case 2.} Let $z \in \DD^{(rt)}$ with $|z| \ge (1-t)^s$. By the 
Bernoulli inequality, $|z| \ge 1-st$. Following \cite{FaGo1}, we write $z = |z| e^{i\theta}$
and
$$
\nu_z(E^{(t)}) = \int_{E^{(t)}} P(z,\xi) \, d\lambda_{\DD}(\xi)
= \bigl(1-|z|^2\bigr) \frac{1}{2\pi} \int_{\{\varphi : e^{i\varphi}\in E^{(t)}\}} 
\frac{d\varphi}{(1-|z|)^2 + 4|z|\sin^2 \frac{\varphi-\theta}{2}}\,. 
$$
Then for $\varphi \in (-\pi\,,\pi]$ with $e^{i\varphi}\in E^{(t)}$,
using $rt \le \dist(z,E) \le 1-|z| + \dist(e^{i\theta},E)$,  
$$
\pi \ge |\phi-\theta| \ge 2 \big|\sin \tfrac{\varphi-\theta}{2}\big| = |e^{i\theta} - e^{i\varphi}|
\ge \dist(e^{i\theta},E) -t \ge rt - (1-|z|) - t \ge \tau\,t\,, 
$$  
where $\tau = r-s-1$. Combining these estimates with \eqref{eq:gtbound},
$$
\begin{aligned}
g^{(t)}(z) &\le 3 \Bigl(\log \frac{1}{1-t}\Bigr) \bigl(1-|z|^2\bigr)\,
\frac{1}{2\pi} \int_{\{\varphi \,:\, \tau\,t \le |\varphi-\theta| \le \pi\}} 
\frac{d\varphi}{(1-|z|)^2 + 4|z|\sin^2 \frac{\varphi-\theta}{2}}\\ 
&= \frac{6}{\pi} \Bigl(\log \frac{1}{1-t}\Bigr)\bigl(1-|z|^2\bigr) \int_{\tau\,t/2}^{\pi/2} 
\frac{d\varphi}{(1-|z|)^2 + 4|z|\sin^2 \varphi}\\
&= \frac{6}{\pi} \Bigl(\log \frac{1}{1-t}\Bigr) 
\arctan \!\left(\frac{1-|z|}{1+|z|} \cot \Bigl(\frac{\tau\, t}{2}\Bigr)\right)\\
&\le \frac{6}{\pi} \Bigl(\log \frac{1}{1-t}\Bigr) \Bigl(\cot \frac{\tau\, t}{2}\Bigr) 
\bigl(1-|z|\bigr)\,.
\end{aligned}
$$
Since $rt <  1 < \pi/3$, we have $\tau\, t /2 < \pi/6$, whence $\cot(\tau\,t/2) \le 2\pi/(3\tau\,t)$.
Also, for $0 < t < 1/r$, we have $\log 1/(1-t) \le r\,t/(r-1)$.
Therefore
$$
g^{(t)}(z) \le \frac{4}{\tau\,t} \Bigl(\log \frac{1}{1-t}\Bigr)\bigl(1-|z|\bigr)
\le \frac{4r}{(r-1)(r-s-1)} \log \frac{1}{|z|}
$$
Thus, in Case 2,
$$
G_{\DD^{(t)}}(z,0) \ge \Bigl(1- \frac{4r}{(r-1)(r-s-1)}\Bigr) G_{\DD}(z,0)\,.
$$
Choosing $r=7$ and $s=18/17$, we get the proposed estimate.
\end{proof}
At the cost of increasing $r$, one can get a better (bigger) lower bound on the disk.
For our purpose, smaller $r_{\DD}$ will be better. The proof allows to take any
number $r > (7 + \sqrt{41})/2$.

\smallskip

With $u$ and $\Psi$ as in Theorem \ref{thm:main1}, we would like to have a
more general type of boundary moment to be finite, even when the integral in
\eqref{eq:fin-int} is infinite. To this end, we consider a continuous, increasing
function $\Phi: \bigl[0\,,\,\diam(\Xx)\bigr] \to [0\,,\,\infty)$ with $\Phi(0)=0$. 
With $\Phi$ as well as with $\Psi$, we associate  the continuous, non-negative measures 
$d\Phi$ and $d\Psi$ on $\bigl(0\,,\,\diam(\Xx)\bigr]$ 
which give mass $\Phi(b)-\Phi(a)$, resp. $\Psi(a)-\Psi(b)$ to any interval 
$(a\,,\,b] \subset \bigl(0\,,\,\diam(\Xx)\bigr]$. Furthermore, we consider the decreasing, 
continuous function
\begin{equation}\label{eq:Upsilon}
\Upsilon: \bigl[0\,,\,\diam(\Xx)\bigr] \to [0\,,\,\infty]\,,\quad 
\Upsilon(t) = \int_t^{\diam(\Xx)} \Phi(s)\,d\Psi(s)\,.
\end{equation} 
It will (typically) occur that $\Upsilon(0)=\infty$. We should consider $\Upsilon$ as a
downscaling of $\Psi$; indeed, $\Upsilon(t) \le \|\Phi\|_{\infty}\, \Psi(t)$.
If $\Psi$ is differentiable on $\bigl(0\,,\,\diam(\Xx)\bigr)$, then $d\Psi(t) = -\Psi'(t)\,dt$, and
$\Upsilon'(t) = \Phi(t) \,\Psi'(t)$. The case considered in \cite{FaGo1} is the one
where $\Psi(t) = t^{-q}$ and $\Phi(t) = t^{\alpha}$, where $0 < \alpha < q$, so that 
$\Upsilon(t) \asymp t^{\alpha-q}$.  

\begin{thm}\label{thm:main2}
Let the subharmonic function $u$ on $\Xx$, the ``singular'' set $E \subset \partial\Xx$ 
and the function $\Psi$ be as in Theorem \ref{thm:main1}, but with infinite 
integral in \eqref{eq:fin-int}. For continuous, increasing
$\Phi: \bigl[0\,,\,\diam(\Xx)\bigr] \to [0\,,\,\infty)$ with $\Phi(0)=0$ and the associated function 
$\Upsilon(t)$ according
to \eqref{eq:Upsilon}, suppose that
\begin{equation*}
\int_{\partial\Xx} \Upsilon \bigl(\dist(\xi,E)\bigr) \,d\lambda(\xi) < \infty\,.
\end{equation*}
Then the Riesz measure $\mu^u$ satisfies the extended boundary moment condition
\begin{equation}\label{eq:momX-Psi}
\int_{\Xx} \dist(x,\partial \Xx)\,\Phi\bigl( \dist(x,E)/R\bigr)\, d\mu^u(x) < \infty\,, 
\end{equation}
where $R =  R_{\Xx}$ is given by $R_{\DD}=14$, resp. $R_{\T}=1$. 
\end{thm}

For the disk case, when $\Psi(t)=t^{-q}$ and $\Phi(t)=t^{\alpha}$ ($0<\alpha < q$),
this boils down to Theorem 1-(ii)-(7) of \cite{FaGo1}.

In typical instances, $\Phi$ will have the \emph{doubling property} 
$\Phi(t/2) \ge C\cdot \Phi(t)$ for a fixed $C > 0$. In this case, division by 
$R$ can be omitted in \eqref{eq:momX-Psi} even on the disk.

\begin{cor}\label{cor:doubling} Consider the disk.
Under the assumptions of Theorem \ref{thm:main2}, if $1/\Psi$ is doubling and
$$
\int_{\Ss} \Psi \bigl(\dsf_{\DD}(\xi,E)\bigr)^{1-\ep} \,d\lambda_{\Ss}(\xi) < \infty\,,
$$ 
then 
$$
\int_{\DD} \dsf_{\DD}(x,\Ss)\,\Psi\bigl( \dsf_{\DD}(x,E)\bigr)^{-\ep}\, d\mu^u(x) < \infty\,. 
$$
\end{cor}
     
\begin{proof}[\bf Proof of Theorem \ref{thm:main2}] Once again, the proof works in
similar ways on disk and tree. We should keep in mind that on the tree,
integrals with respect to the Riesz measure are infinite sums. 

\smallskip

For most of the proof, we assume that $u(o)$ is finite. On the tree, this is always required,
but on the disk, one may have $u(z) = -\infty$ on a set of measure $0$. We shall briefly
explain at the end how to handle the case $u(0)= -\infty$.   

\smallskip

We take up the thread from the end of the proof of Theorem \ref{thm:main1}, in particular
\eqref{eq:u-le-ht}. That inequality tells us that
$u$ has $c_{\Xx}\, h^{(t)}$ as a harmonic majorant on $\Xx^{(t)}$. Thus, it
has its least harmonic majorant $v^{(t)}$ on that set, and we have the Riesz decomposition
$$
u(x) = v^{(t)}(x) - G_{\Xx^{(t)}}\mu^u(x)\,,\quad x \in \Xx^{(t)}.
$$ 
We have $G_{\DD}(z,0) \ge 1-|z|= \dsf_{\DD}(z,\Ss)$ on the disk, and 
$G_{\T}(x,o) = b_{\T}\, \rho_{\T}(x,\partial\T)$. 
Using Theorem \ref{thm:green}, we get for $0 < t < 1/r$ ($r = r_{\Xx}$)
$$
\begin{aligned}
\int_{\Xx^{(rt)}} \dist(x,&\partial\Xx)\,d\mu^u(x) 
\le b_{\Xx}\, G_{\Xx^{(t)}}\mu^u(o) \\ 
&= b_{\Xx}\, \bigl(v^{(t)}(o) - u(o)\bigr)
\le  b_{\Xx}\,c_{\Xx}\, h^{(t)}(o) - b_{\Xx}\,u(o)\\
&= b_{\Xx}\,c_{\Xx}\,
\int_{E^{(t)}_*}\Psi\bigl( \dist(\cdot,E)\bigr)\, d\lambda
+ b_{\Xx}\,c_{\Xx}\, \Psi(t)\, \lambda(E^{(t)}) - b_{\Xx}\,u(o).  
\end{aligned}
$$
(In the disk case, $o$ stands once more for the origin.) For the next computation,
we note that $\max \{ \dist(x,E) : x  \in \Xx \}$ has value $1$ for the tree, but may
be between $1$ and $2$ for the disk. Tacitly using continuity of the involved measures, 
and using monotonicity of $\Psi$, for $0 < t < 1$
$$
\begin{aligned}
\int_{E^{(t)}_*} &\Psi\bigl( \dist(\xi,E)\bigr)\, d\lambda(\xi)
=  \int_{E^{(1)} \cap E^{(t)}_*} \Psi\bigl( \dist(\xi,E)\bigr)\, d\lambda(\xi)
+ \int_{E^{(1)}_*} \Psi\bigl( \dist(\xi,E)\bigr)\, d\lambda(\xi)\\
&\le 
\int_{E^{(1)} \cap E^{(t)}_*}\int_{\dist(\xi,E)}^1 \,d\Psi(s)\, d\lambda(\xi)
\;\;+\;\; \Psi(1) \, \lambda(E^{(1)} \cap E^{(t)}_*) + \Psi(1) \, \lambda(E^{(1)}_*)\\
&= \int_{t}^1 
\lambda\bigl(\{ \xi \in \partial \DD: t < \dist(\xi,E) \le s \}\bigr)\,d\Psi(s)
\;\;+\;\ \Psi(1) \, \lambda(E^{(t)}_*)) \\
&= \int_{t}^1 \lambda(E^{(s)})\,d\Psi(s)
 \;\;-\;\; \lambda(E^{(t)})\,\Psi(t) \;+\; \Psi(1)\,.
\end{aligned}
$$
Combining this with the previous inequality, 
we get for $0 < t < 1$
\begin{equation}\label{eq:ineq2}
\int_{x \in\Xx^{(t)}} \dist(x,\partial\Xx)\,d\mu^u(x) 
\le b_{\Xx}\,c_{\Xx}\, \int_{t/r}^1 \lambda(E^{(s)})\,d\Psi(s) + C_1\,, 
\end{equation}
where $\;C_1 =  b_{\Xx}\,c_{\Xx}\,\Psi(1) - b_{\Xx}\,u(o)$. 
Because of several smaller subtleties, we now conclude the proofs separately.
\\[5pt]
\textbf{A. Tree case.} 
Recalling that $b_{\T}= r_{\T} = R_{\T} = 1$,
$$
\begin{aligned}
\sum_{x \in \T} &\rho_{\T}(x,\partial \T) \, \Phi\bigl(\rho_{\T}(x,E)\bigr)\, \mu^u(x) 
= \sum_{x \in \T} \rho_{\T}(x,\partial \T) \int_0^{\rho_{\T}(x,E)} d\Phi(t) \, \mu^u(x) \\
&= \int_0^{1} \Biggl( \sum_{\,x \in \T^{(t)}} \rho_{\T}(x,\partial \T) \,\mu^u(x) \Biggr) 
d\Phi(t)
\\ [\text{by \eqref{eq:ineq2}}]\qquad 
&\le c_{\T} \int_0^{1} \int_t^1 \lambda_{\T}(E^{(s)})\,d\Psi(s) \,d\Phi(t)
+ C_1 \,\Phi(1) \\
[\text{Fubini}]\qquad 
&= c_{\T} \int_0^1 \lambda_{\T}(E^{(s)}) \,\Phi(s)\,d\Psi(s) + C_2
= c_{\T} \int_{\partial \T} \Upsilon\bigl( \rho_{\T}(\xi,E) \bigr)\,d\lambda_{\T}(\xi) + C_2\,,
\end{aligned} 
$$
which is finite by assumption.
\\[5pt]
\textbf{B. Disk case.} Note that the maximum possible value of $\dsf_{\DD}(z,E)$ is $2$. 
We refer to a simple observation of \cite{FaGo1}: if $0 < t < 2$ then 
for every $z \in \DD$ and $\alpha \in [0\,,\,1]$, we have 
$\dsf_{\DD}(z,E) \le 2 \dsf_{\DD}(\alpha\,z, E)$. In particular, if $\dsf_{\DD}(z,E) > t$ 
then $\dsf_{\DD}(\alpha\,z, E) > t/2$, so that $z$ lies in the component of $0$ of the
set $\{ w \in \DD :  \dsf_{\DD}(w,E) > t/2\}$. This means that
\begin{equation}\label{eq:subset}
\{ z \in \DD :  \dsf_{\DD}(z,E) > t\}
\subset \DD^{(t/2)}\,.
\end{equation}
Using this, we now compute
$$
\begin{aligned}
\int_{\DD} \dsf_{\DD}(z,\Ss)\, \Phi\bigl( \dsf_{\DD}(z,E)/14\bigr) \,d\mu^u(z)
&= \int_{\DD} \,\, \int_0^{\dsf_{\DD}(z,E)/14} \dsf_{\DD}(z,\Ss)\,d\Phi(t) \,d\mu^u(z)\\
[\text{since}\; \dsf(z,E) < 2]\qquad 
&= \int_0^{1/7} \int_{\{ z \in \DD \,:\,  \dsf_{\DD}(z,E) > 14t\}}  
\dsf_{\DD}(z,\Ss)\,d\mu^u(z)\,d\Phi(t)\\
\\ [\text{by \eqref{eq:subset}}]\qquad 
&\le \int_0^{1/7} \int_{\DD^{(7t)}}  \dsf_{\DD}(z,\Ss)\,d\mu^u(z)\,d\Phi(t)\\
\\ [\text{by \eqref{eq:ineq2}}]\qquad 
&\le b_{\Xx}\,c_{\Xx}\,\int_0^1 \int_{t}^1 \lambda(E^{(s)})\,d\Psi(s) + 
C_1\,\Phi(1)\,,
\end{aligned}
$$
which is seen to be finite by the same calculation as in the tree case.

\smallskip

The case when $u(0) =-\infty$ can be treated exactly as in \cite[p.43]{FaGo1} (where
the subharmonic function is denoted $v$) and is omitted here.
\end{proof}

Finally, we want to prove a converse to Theorem \ref{thm:main2} analogous to 
Proposition \ref{pro:converse}.

\begin{thm}\label{thm:converse}
Let the set $E \subset \partial\Xx$  and the function $\Psi$ be as in Theorem \ref{thm:main1}, 
but with infinite  integral in \eqref{eq:fin-int}. Let 
$\Phi: [0,1] \to [0\,,\,\infty)$ be continuous and increasing with $\Phi(0)=0$ and $\Phi(t) > 0$ for
$t> 0$. For the associated function $\Upsilon(t)$ according
to \eqref{eq:Upsilon}, suppose that
\begin{equation*}
\int_{\partial\Xx} \Upsilon \bigl(\dist(\xi,E)\bigr) \,d\lambda(\xi) = \infty\,.
\end{equation*}
If $u$ is a subharmonic function on $\Xx$ such that 
$$
u(x) \ge \Psi\bigl(\dist(x,E)\bigr)
$$
then the Riesz measure $\mu^u$ is such that
\begin{equation}\label{eq:inf-momX-Phi}
\int_{\Xx} \dist(x,\partial \Xx)\,\Phi\bigl( \dist(x,E)\bigr)\, d\mu^u(x) = \infty\,. 
\end{equation}
\end{thm}

\begin{proof} First of all, we note that \eqref{eq:inf-momX-Phi} hold if and only if
\begin{equation}\label{eq:inf-pot-Phi}
\int_{\Xx} G(x,o)\,\Phi\bigl( \dist(x,E)\bigr)\, d\mu^u(x) = \infty\,. 
\end{equation}
On the tree, this is obvious, because $G_{\T}(x,o) = \frac{\qq}{\qq-1} \,\rho_{\T}(x,\partial\T)$.
On the disk, it is clear that \eqref{eq:inf-momX-Phi} implies \eqref{eq:inf-pot-Phi}. Conversely,
$$
\int_{|z| < 1/2}  G(z,0)\,\Phi\bigl( \dsf_{\DD}(z,E)\bigr)\, d\mu^u(z) 
\le \|\Phi\|_{\infty} \int_{|z| < 1/2}  G(z,0)\, d\mu^u(z) < \infty\,,
$$
while for $|z| \ge 1/2$, we have $G(z,0) = \log \frac{1}{|z|} \le (2\log 2) (1-|z|)$,
so that \eqref{eq:inf-pot-Phi} implies
$$
2\log 2 \int_{|z| \ge 1/2}  (1-|z|)\,\Phi\bigl( \dsf_{\DD}(z,E)\bigr)\, d\mu^u(z) 
\ge \int_{|z| \ge 1/2} G(z,0)\,\Phi\bigl( \dsf_{\DD}(z,E)\bigr)\, d\mu^u(z) = \infty\,.
$$
\emph{Case 1.} Suppose that there is $t \in (0\,,\,1)$ such that $u$ has no harmonic
majorant on the set $\Xx^{(t)}$. Then $G_{\Xx^{(t)}}\mu^u$ is infinite on that set.
Thus, 
$$
\begin{aligned}
\int_{\Xx} G(x,o) \,\Phi\bigl( \dist(x,E)\bigr)\, d\mu^u(x)
&\ge \int_{\Xx^{(t)}} G_{\Xx^{(t)}}(x,o) \,\Phi\bigl( \dist(x,E)\bigr)\, d\mu^u(x)\\
&\ge \Phi(t) \,G_{\Xx^{(t)}}\mu^u(o) = \infty\,,
\end{aligned}
$$
and the equivalence of \eqref{eq:inf-momX-Phi} with \eqref{eq:inf-pot-Phi} implies 
the result.
\\[5pt]
\emph{Case 2.} We are left with the case when for each $t \in (0\,,\,1)$ there is
the (finite) least harmonic majorant $v^{(t)}$ of $u$ on $\Xx^{(t)}$. Recall the function
$h^{(t)}$ of \eqref{eq:ht}. Then for every $\eta \in \partial \Xx^{(t)}$,
$$
\limsup_{x \to \eta} v^{(t)}(x) \ge \limsup _{x \to \eta} u(x) \ge \Psi\bigl(\dist(\eta,E)\bigr)
= \lim_{x \to \eta} h^{(t)}(x)\,.
$$
By the minimum principle, applied to the harmonic function $v^{(t)} - h^{(t)}\,$,
we have $v^{(t)} \ge h^{(t)}$ on $\Xx^{(t)}$. Now we can replace the computations of the proof
of Theorem \ref{thm:main2} with similar inequalities in the reverse direction.
$$
\begin{aligned}
 \int_{\Xx^{(t)}} &G(x,o) \,d\mu^u(x) \ge G_{\Xx^{(t)}}\mu^u(o) = v^{(t)}(o) - u(o)
\ge h^{(t)}(o) - u(o)\\
&= \int_{E^{(1)} \cap E^{(t)}_*} \Psi\bigl(\dist(\cdot,E)\bigr)\,d\lambda 
   + \int_{E^{(1)}_*} \Psi\bigl(\dist(\cdot,E)\bigr)\,d\lambda + \Psi(t)\,\lambda(E^{(t)}) - u(o) \\
&\ge \int_t^1 \lambda(E^{(s)} \setminus E^{(t)}) \,d\Psi(s) 
+ \Psi(1)\,\lambda(E^{(1)} \setminus E^{(t)}) + \Psi(1)\,\lambda(E^{(1)}_*)
+ \Psi(t)\,\lambda(E^{(t)}) - u(o)\\
&= \int_t^1 \lambda(E^{(s)}) \,d\Psi(s) + C_3\,, \hspace*{2cm}\text{where} \quad C_3 = \Psi(1)-u(o)\,.
\end{aligned}
$$
Now let $0 < \ep < 1$. Let $\Phi_{\ep}(s) = \max\{\Phi(s)-\Phi(\ep)\,,\, 0\}$.
Since $u$ has a harmonic majorant on $\Xx^{(\ep)}$, the first integral in the following 
computation is finite. The above estimate is used in the third line.
$$
\begin{aligned}
\int_{\Xx^{(\ep)}} G(x,o) \,&\Phi\bigl(\dist(x,E)\bigr)\, d\mu^u(x) 
\ge \int_{\Xx^{(\ep)}} G(x,o) \int_{\ep}^{\dist(x,E)} \,d\Phi(t)\, d\mu^u(x) \\
&\ge \int_{\ep}^1 \int_{\Xx^{(t)}} G(x,o)\, d\mu^u(x)\, d\Phi(t) \\
&\ge \int_{\ep}^1 \int_t^1 \lambda(E^{(s)}) \, d\Psi(s)\, d\Phi(t) + (1-\ep)C_3 \\
&= \int_{\ep}^1 \lambda(E^{(s)}) \int_\ep^s d\Phi(t) \, d\Psi(s) +  (1-\ep)C_3 \\
&=\int_0^1 \biggl(\int_{\{\xi \in \partial  \Xx : \dist(\xi,E) \le s\}} d\lambda(\xi)\biggr) 
\Phi_{\ep}(s)\,d\Psi(s) + (1-\ep)C_3 \\
&= \int_{E^{(1)}} \int_{\dist(\xi,E)}^1 \Phi_{\ep}(s) \, d\Psi(s) \, d\lambda(\xi) + (1-\ep)C_3 \\
\end{aligned}
$$
As $\ep \to 0$, by monotone convergence, the double integral in the last line tends to
$$
\int_{E^{(1)}} \Bigl(\Upsilon\bigl(\dist(\xi,E)\bigr) - \Upsilon(1) \Bigr)\, d\lambda(\xi)\,,
$$ 
which is infinite by assumption.
\end{proof}

\begin{rmks}\label{rmk:hyp} 
(a) \emph{[Hyperbolic versus Euclidean.]} 
In the introduction and in Section \ref{sec:disk-tree} we insisted on a hyperbolic  
``spirit'' inherent in the material presented here. After all, this was not dominant
in most of our computations. Not only on the disk, we always used the Euclidean metric
$\dsf_{\DD}$, but also on the tree, the dominant role was played by the metric
$\rho_{\T}$ which is the tree-analogue of the Euclidean metric. One point is that
to see the latter analogy, one should first understand that the graph metric on the tree
corresponds to the hyperbolic one on the disk.   

One result where hyperbolicity is strongly present is Theorem \ref{thm:green}. The proof in
the tree case relies directly on the fact that the tree with its graph metric is
$\delta$-hyperbolic in the sense of {\sc Gromov}~\cite{Gro}, with $\delta=0$: every
vertex is a cut-point (it disconnects the tree). Analogously, one might try to prove that
theorem in the disk case using $\delta$-hyperbolicity with $\delta = \log(1 + \sqrt{2}\,)$. 
Indeed, this is related
with the inequalities of {\sc Ancona}~\cite{Anc} which say that the Green kernel of
the open disk is almost submultiplicative along hyperbolic geodesics. (For the disk, this can
be seen by direct inspection via the explicit formulas for the Green kernel.) Now, for 
points $z \in \DD^{(rt)}$ and $\xi \in E^{(t)}$, the hyperbolic geodesic from $z$ to $\xi$
must be at bounded hyperbolic distance from the origin (depending on $r$ and $t$), 
similarly to the (simpler) tree case. However, this idea is more vague than the 
down-to-earth proof following \cite{FaGo1}.
\\[5pt]
(b) In view of the equivalence \eqref{eq:inf-momX-Phi} $\iff$ \eqref{eq:inf-pot-Phi},
in all the results presented here, one can replace the distance to the boundary
$\dist(x,\partial \Xx)$ with the Green kernel $G(x,o)$.
\\[5pt]
(c) Among the common features of disk and tree which allowed us to formulate and
prove the results in very similar ways, the key facts are 
\begin{itemize}
 \item comparability of $G(x,o)$ with $\dist(x,\partial\Xx)$ (the metric is ``intrinsic''
in this sense),
 \item solvability of the Dirichlet problem for continuous functions on $\partial\Xx$, 
and in particular, vanishing of the Green kernel at the boundary, and
 \item the Green kernel estimate of Theorem \ref{thm:green}.
\end{itemize}
\end{rmks}

\begin{imp}{\bf An extension for trees.} 
 Instead of the homogeneous tree, we can take an arbitrary locally finite tree $\T$ 
and equip its edges with \emph{conductances} $a(x,y) = a(y,x) > 0 \iff x \sim y$. Letting
$m(x) = \sum_y a(x,y)$, the transition probabilities $p(x,y) = a(x,y)/m(y)$ give rise
to a nearest neighbour random walk $(Z_n)_{n \ge 0}$ and to the associated Laplacian
$$
\Lap_{\T} f(x) = \sum_{y \sim x} p(x,y) \bigl(f(y)-f(x)\bigr).
$$
We asssume the following. 
\begin{itemize}
 \item[(i)] Strong irreducibility:  $\; 0 < m_0 \le m(x) \le M_0 < \infty\;$ and $\;a(x,y) \ge a_0 > 0\;$
for all $x$ and all $y \sim x$.
\item[(ii)] Strong transience: $\; F(x,y) \le \delta < 1\;$ for all $x$ and all $y \sim x$,
where for arbitrary $x,y \in \T$,
$$
F(x,y) = \Prob[\exists n \ge 0 : Z_n = y \mid Z_0=x]
$$ 
\end{itemize}
The associated Green kernel
$$
G(x,y) = \sum_{n=0}^{\infty} p^{(n)}(x,y)\,,\quad \text{where}\quad
p^{(n)}(x,y) = \Prob[Z_n=y \mid Z_0=x]\,,\quad x,y \in X
$$
is finite and tends to $0$ at infinity by assumption (ii). Note that in our 
notation, $G(x,y)=F(x,y)G(y,y)$. 

We can adapt all the above results regarding the homogenous tree to this more
general situation. The main issue is to define a suitable metric on the compactification
$\wh \T$ in the right way: for $z, w \in \wh \T$,
$$
\rho_{\T}(w, z) = \begin{cases} F(w \wedge z,o) \,,&\text{if}\; z\ne w\,,\\
                                              0 \,,&\text{if}\; z=w\,.
\end{cases}
$$
[For simple random walk on the homogeneous tree, as considered above, this is just the 
metric of \eqref{eq:ultra}.]

In this setting, the tree-versions of theorems \ref{thm:main1}, \ref{thm:main2} and
\ref{thm:converse} remain true.
This applies, in particular, to arbitrary symmetric nearest neighbour random walks on the
free group ($\equiv$ homogeneous tree with even degree). 
\end{imp}

In conclusion, we remark that the very recent note by {\sc Favorov and Radchenko}~\cite{FaRa}
was written in parallel to the present article without mutual knowledge. The results of
\cite{FaRa} concern the disk case and are a bit less general than ours.
We want to point out that
here, our main focus has been on elaborating some aspects of the very strong analogies of the
potential theory on disk and tree, respectively, via focussing on properties of Riesz measures.

\smallskip

\textbf{Acknowledement.} The authors acknowledge email exchanges with M. Stoll (Columbia, SC)
and with S. Favorov (Kharkov).

\end{document}